\documentclass[12pt]{amsart}

\usepackage{amsmath, amssymb,graphicx}
\usepackage[pdftex, plainpages=false]{hyperref}

\newtheorem{theorem}{Theorem}
\newtheorem{lemma}[theorem]{Lemma}
\newtheorem{proposition}[theorem]{Proposition}
\newtheorem{corollary}[theorem]{Corollary}
\newtheorem*{claim}{Claim}
\theoremstyle{definition}
\newtheorem{definition}[theorem]{Definition}
\theoremstyle{remark}

\newtheorem{problem}{Problem}

\newcommand {\Hyp}{\mathbb{H}} 
\newcommand {\Parabolics}{\mathbb{P}}

\newcommand {\mc}{\mathcal} 

\DeclareMathOperator{\isom}{\mathsf{Isom}}

\DeclareMathOperator{\dist}{\mathsf{dist}}

\begin{document}

\title[Virtual Quasiconvex Amalgamation]{Virtual Amalgamation of Relatively Quasiconvex Subgroups}
\author[E.~Mart\'inez-Pedroza]{Eduardo Mart\'inez-Pedroza}
      \address{Memorial University\\
Saint John's, Newfoundland, Canada}
      \email{emartinezped@mun.ca}

\author[A.~Sisto]{Alessandro Sisto}
      \address{
                     Mathematical Institute \\
               Oxford OX1 3LB, United Kingdom }
      \email{sisto@maths.ox.ac.uk}
\subjclass[2000]{}
\date{\today}

\begin{abstract}
For relatively hyperbolic groups, we investigate conditions guaranteeing that the subgroup generated 
by two relatively quasiconvex subgroups $Q_1$ and $Q_2$ is relatively quasiconvex and isomorphic to $Q_1 \ast_{ Q_1 \cap Q_2} Q_2$. The main theorem extends results for quasiconvex subgroups of word-hyperbolic groups,  and results for discrete subgroups of isometries of hyperbolic spaces.
\end{abstract}
\maketitle

\section{Introduction}

This paper continues the work that started in~\cite{MP09} motivated by the following question: 
\begin{problem}
Suppose $G$ is a relatively hyperbolic group, $Q_1$ and $Q_2$ are relatively quasiconvex subgroups of $G$. 
Investigate conditions guaranteeing that the natural homomorphism
\begin{eqnarray*}  Q_1 \ast_{Q_1\cap Q_2} Q_2 \longrightarrow G  \end{eqnarray*}
is injective and that its image $\langle Q_1\cup Q_2\rangle$ is relatively quasiconvex. 
\end{problem}

Let $G$ be a group hyperbolic relative to a finite collection of subgroups $\Parabolics$, and let $\dist$ be a proper left invariant metric on $G$.
\begin{definition}
Two subgroups $Q$ and $R$ of $G$ have \emph{compatible parabolic subgroups} if for any maximal parabolic subgroup $P$ of $G$ either $Q\cap P < R\cap P$ or $R\cap P < Q\cap P$.
\end{definition}

\begin{theorem}\label{main}
For any pair of  relatively quasiconvex subgroups $Q$ and $R$ of $G$, there is a constant $M=M(Q, R, \dist)  \geq 0$
with the following property.  Suppose that $Q' < Q$ and $R' < R$ are subgroups such that
\begin{enumerate}
\item $Q' \cap R'$ has finite index in $Q \cap R$,
\item $Q'$ and $R'$ have compatible parabolic subgroups, and
\item $\dist( 1, g) \geq M$ for any $g$ in  $Q' \setminus Q'\cap R'$ or $R' \setminus Q'\cap R'$.
\end{enumerate}
Then the subgroup $\langle Q' \cup R' \rangle$ of $G$ satisfies:
\begin{enumerate}
\item The natural homomorphism \begin{eqnarray*}  Q' \ast_{Q'\cap R'} R' \longrightarrow \langle Q' \cup R' \rangle   \end{eqnarray*} is an isomorphism. 
\item If $Q'$ and $R'$ are relatively quasiconvex, then so is $\langle Q' \cup R' \rangle$.
\end{enumerate}
\end{theorem}

Theorem~\ref{main} extends results by Gitik \cite[Theorem 1]{Gi99} for word-hyperbolic groups and by the the first author~\cite[Theorem 1.1]{MP09} for relatively hyperbolic groups, as well as the case of \cite[Theorem 5.3]{BC05} when $\Gamma$ is geometrically finite.

\begin{definition}
Two subgroups $Q$ and $R$ of a group $ G$ can be \emph{virtually  amalgamated} if there are finite index subgroups $Q' < Q$ and $R' < R$ such that the natural map $Q' \ast_{Q'\cap R'} R' \longrightarrow G$ is injective.  
\end{definition}

Let $Q$ and $R$ be relatively quasiconvex subgroups of $G$, and let $M$ be the constant provided by Theorem~\ref{main}. If either $G$ is residually finite or $Q\cap R$ is a separable subgroup of $G$, then there is a finite index subgroup $G'$ of $G$ such that $\dist (1, g) > M$ for every $g\in G$ with $g \not \in Q\cap R$. In the case that there is such subgroup $G'$,  then the subgroups $Q'=G'\cap Q$ and $R'=G'\cap R$ are relatively quasiconvex and satisfy the hypothesis of Theorem~\ref{main}; hence they have a quasiconvex virtual amalgam. 

\begin{corollary}[First Virtual Quasiconvex Amalgam Theorem]
Let $Q$ and $R$ quasiconvex subgroups of $G$ with compatible parabolic subgroups, and suppose that $Q\cap R$ is separable.  Then $Q$ and $R$ can be virtually amalgamated in $G$.
\end{corollary}

\begin{corollary}[Second Virtual Quasiconvex Amalgam Theorem]
Suppose that $G$ is residually finite.  Then any pair of relatively quasiconvex subgroups with compatible parabolic subgroups has a quasiconvex virtual amalgamam. 
\end{corollary}

An immediate corollary of the Virtual Quasiconvex Amalgam Theorem for residually finite relatively hyperbolic groups provides the following result by Baker-Cooper~\cite[Theorem 5.3]{BC05}.

\begin{corollary}[GF subgroups have virtual amalgams.] Suppose that $G$ is a geometrically finite subgroup of $\isom(\Hyp^n)$. If $Q$ and $R$ are subgroups of $G$ with compatible parabolic subgroups, then $Q$ and $R$ have a virtual amalgam. The resulting subgroup is geometrically finite if $Q$ and $R$ are geometrically finite.
\end{corollary}

\section{Preliminaries} \label{sec:preliminaries}

\subsection{Gromov-hyperbolic Spaces}

Let $(X, \dist)$ be a proper and geodesic $\delta$-hyperbolic space. Recall that a $(\lambda,\mu)-$quasi-geodesic is a curve $\gamma\colon [a, b] \to X$ parametrize by arc-length such that 
\[ |x-y|/\lambda -\mu  \leq \dist(\gamma(x), \gamma(y)) \leq \lambda|x-y|+\mu \]
for all $x,y\in [a,b]$. The curve $\gamma$ is a $k-$local $(\lambda,\mu)-$quasi-geodesic if the above condition is required only for $x,y\in [a,b]$ such that $|x-y|\leq k$.

\begin{lemma}\cite[Chapter 3, Theorem 1.2]{CDP-ggt}(Morse Lemma)
\label{morse}
 For each $\lambda,\mu,\delta$ there exists $k>0$ with the following property. In an $\delta-$hyperbolic geodesic space, any $(\lambda, \mu)-$quasi-geodesic at $k$-Hausdorff-distance from the geodesic between its endpoints.
\end{lemma}

\begin{lemma}\cite[Chapter 3, Theorem 1.4]{CDP-ggt}
\label{concat}
 For each $\lambda,\mu,\delta$ there exist $k,\lambda',\mu'$ so that any $k-$local $(\lambda,\mu)-$quasi-geodesic in a $\delta-$hyperbolic geodesic space is a $(\lambda',\mu')-$quasi-geodesic.
\end{lemma}

 Fix a basepoint $x_0 \in X$. If $G$ is a subgroup of $Isom (X)$, we identify each element $g$ of $G$ with the point $gx_0$ of $X$.  For $g_1,g_2 \in G$ denote by $\dist (g_1, g_2)$ the distance $\dist (g_1 x_0, g_2 x_0)$. Observe that if $G$ is a discrete subgroup, this is a proper and left invariant pseudo-metric on $G$. 

\begin{lemma}[Bounded Intersection] \cite[Lemma 4.2]{MP09} \label{lem:bounded-intersection}
Let $G$ be a discrete subgroup of $\isom (X)$, let $Q$ and $R$ be subgroups of $G$, and let $\mu>0$ be a real number.  Then there is a constant $M=M(Q, R, \mu) \geq 0$ so that
\begin{equation*}  Q \cap \mc N_{\mu}(R)  \subset  \mc N_{M} ( Q \cap R). \end{equation*}
\end{lemma}

\subsection{Relatively Quasiconvex Subgroups}

We follow the approach to relatively hyperbolic groups as developed by Hruska~\cite{HK08}. 

\begin{definition}[Relative Hyperbolicity] \label{defn:rel-hyp-gp} A group $G$ is {\em relatively hyperbolic with respect to a finite collection of subgroups $\Parabolics$} if $G$ acts properly discontinuously and by isometries on a proper and geodesic $\delta$-hyperbolic space $X$ with the following property: $X$ has a $G$-equivariant collection of pairwise disjoint horoballs whose union is an open set $U$, $G$ acts cocompactly on $X-U$, and $\Parabolics$ is a set of representatives of the conjugacy classes of parabolic subgroups of $G$.
\end{definition}

Throughout the rest of the paper, $G$ is a relatively hyperbolic group acting on a proper and geodesic $\delta$-hyperbolic space $X$ with a $G$-equivariant collection of  horoballs satisfying all conditions of Definition~\ref{defn:rel-hyp-gp}. As before, we fix a basepoint $x_0 \in X-U$,  identify each element $g$ of $G$ with $gx_0\in X$ and let $\dist (g_1, g_2)$ denote $\dist (g_1 x_0, g_2 x_0)$ for $g_1,g_2 \in G$. 

\begin{lemma}\cite[Lemma 6.4]{Bow-99-rel-hyp}(Cocompact actions of parabolic subgroups on thick horospheres)\label{lem:parabolic-action}
Let $B$ be a horoball of $X$ with $G$-stabilizer $P$.  For any $M>0$,  $P$ acts cocompactly on $\mc N_M (B ) \cap (X-U)$.
\end{lemma}

\begin{lemma}[Parabolic Approximation]\label{lem:parabolic-approximation-2}
Let $Q$ be a subgroup of $G$ and let $\mu>0$ be a real number.  There is a constant $M=M(Q, \mu)$ with the following property. If $P$ is a maximal parabolic subgroup of $G$ stabilizing a horoball $B$,   and $\{1, q\} \subset Q \cap \mc N_\mu (B)$ then there is $p \in Q\cap P$ such that $\dist (p, q)<M$.
\end{lemma}
\begin{proof}
By Lemma~\ref{lem:parabolic-action},  $\dist (q, P) <M_1$ for some constant $M_1=M_1(Q, P)$. Then Lemma~\ref{lem:bounded-intersection} implies that $\dist (q, Q\cap P) <M_2$ where $M_2=N(Q, P, M_1)$. Since $B$ is a horoball at distance less than $\mu$ from $1$, there are only finitely many possibilities for $B$ and hence for the subgroup $P$. Let $M$ the maximum of all $N(Q, P, \mu)$ among the possible $P$.
\end{proof}

\begin{definition}[Relatively Quasiconvex Subgroup] A subgroup $Q$ of $G$ is {\em relatively quasiconvex} if there is $\mu \geq 0$ such that for any geodesic $c$ in $X$ with endpoints in $Q$, $c \cap (X - U) \subset N_\mu (Q).$
\end{definition}

The choice of horoballs turns out not to make a difference:      

\begin{proposition}\cite{HK08}\label{prop}
If $Q$ is relatively quasiconvex in $G$ then for any $L\geq 0$ there is $\mu \geq 0$ such that for any geodesic $c$ in $X$ with endpoints in $Q$, $c \cap \mc N_L(X - U) \subset N_\mu (Q).$
\end{proposition}

\section{Proof of the main theorem}
For the convenience of the reader, we report below the statement of Theorem \ref{main}.

\begin{theorem}
For any pair of  relatively quasiconvex subgroups $Q$ and $R$ of $G$, there is a constant $M=M(Q, R, \dist)  \geq 0$
with the following property.  Suppose that $Q' < Q$ and $R' < R$ are subgroups such that
\begin{enumerate}
\item $Q' \cap R'$ has finite index in $Q \cap R$,
\item $Q'$ and $R'$ have compatible parabolic subgroups, and
\item $\dist( 1, g) \geq M$ for any $g$ in  $Q' \setminus Q'\cap R'$ or $R' \setminus Q'\cap R'$.
\end{enumerate}
Then the subgroup $\langle Q' \cup R' \rangle$ of $G$ satisfies:
\begin{enumerate}
\item The natural homomorphism \begin{eqnarray*}  Q' \ast_{Q'\cap R'} R' \longrightarrow \langle Q' \cup R' \rangle   \end{eqnarray*} is an isomorphism. 
\item If $Q'$ and $R'$ are relatively quasiconvex, then so is $\langle Q' \cup R' \rangle$.
\end{enumerate}
\end{theorem}

Consider $1\neq g\in Q' \ast_{Q'\cap R'} R'$ and write it as $g=g_1\dots g_n$ where the $g_i$'s are alternatively elements of $Q'\backslash Q'\cap R'$ and $R'\backslash Q'\cap R'$. Moreover, assume that this product is \emph{minimal} in the sense that $\sum \dist (1,g_i)$ is minimal among all such products describing $g$.
\begin{claim}(Lemma~\ref{lem:claim-2} below). There is a constant $K=K(Q,R,\delta)$ with the following property. For each $i$, let $h_i=g_1\dots g_i$.  Then the concatenation $\alpha = \alpha_1\cdots \alpha_{n-1}$ of geodesics $\alpha_i$ from $h_i$ to $h_{i+1}$ is an $M'-$local $(1,K)-$quasi-geodesic for $M'= \min\{ \dist (1, g_i) \}$.
\end{claim}

\subsection*{Conclusion of the proof using the claim.} If we require $M$ as in the statement of the theorem to be large enough, then we can assume $M'> k,\lambda'\mu'$ where $k,\lambda'$ and $\mu'$ are as in Lemma~\ref{concat} for $\lambda=1, \mu=K$.
It follows that $\alpha$ is a quasi-geodesic with distinct endpoints, and hence $g\neq 1$ in $G$. Therefore we have shown that the map $Q'\ast_{Q'\cap R'}R' \to G$ is injective.

It is left to prove that if $Q'$ and $R'$ are relatively quasiconvex, then $\langle Q', R'\rangle$ is relatively quasiconvex. By Lemma \ref{morse} (Morse Lemma), any $(\lambda', \mu')-$quasi-geodesic is at Hausdorff distance at most $L$ from any geodesic between its endpoints. In particular, if $\gamma$ is a geodesic from $1$ to $g$, then $\gamma \cap (X-U) \subseteq \mc N_{L}(\alpha )\cap (X-U)$. It is enough to show that $\alpha \cap \mc N_{L}(X-U)$ is contained in $\mc N_{\mu} (\langle Q'\cup R'\rangle )$.  Let $p\in\alpha \cap \mc N_{L} (X-U)$ and let $i$ be so that $p\in [h_i,h_{i+1}]\cap (X-U)$. Assume $g_{i+1}\in Q'$, the other case being symmetric. As $Q'$ is relatively quasiconvex and in view of Proposition~\ref{prop}, there is a constant $\mu$ so that $p\in \mc N_\mu(h_iQ')\subseteq \mc N_\mu( \langle Q'\cup R'\rangle )$ (as $h_i\in \langle Q'\cup R'\rangle$). 

\subsection*{Proof of the Claim.} The proof is a sequence of three lemmas. 

\begin{lemma}\label{lem:computation}
Suppose $a\in Q'\cap R'$, $p$ is a point at distance at most $\delta$ from the geodesic segment $[1, g_ig_{i+1}]$ and $\dist (p,g_ia)\leq M$. Then 
\[\dist (1,g_i)+\dist (1,g_{i+1})   \leq  \dist (1, g_ig_{i+1})+2M+2\delta.\] 
\end{lemma}
\begin{proof}
Let $p' \in [1, g_ig_{i+1}]$ be such that $\dist (p, p')<\delta$. Then
\begin{equation} \begin{split} 
\dist (1,g_ia)+& \dist (1,a^{-1}g_{i+1}) \leq \\
& \leq \dist (1, p') + \dist (p', g_ia) + \dist (g_i a, p') + \dist (p', g_ig_{i+1}) \\ 
& \leq \dist (1, g_ig_{i+1}) + 2M+2\delta \nonumber \end{split}  \end{equation}
As $g$ can be written as $g_1\dots (g_ia)(a^{-1}g_{i+1})\dots g_n$, the minimality assumption implies
$\dist (1,g_i)+\dist (1,g_{i+1})   \leq  \dist (1, g_ig_{i+1})+2M+2\delta. 
$
\end{proof}

\begin{lemma}(Gromov's Inner Product is Bounded)\label{lem:claim} 
There exists a constant $K=K(Q, R)$, not depending on $g$, such  that  \[\dist (1,g_i)+\dist (1,g_{i+1})  \leq \dist (1,g_{i}g_{i+1}) + K .\]
\end{lemma}
\begin{proof}
Constants which depend only on $Q$, $R$ and $\delta$ are  denoted by $M_i$, the index counts positive increments of the constant during the proof. The constant $K$ of the statement corresponds to $M_{11}$.

Suppose $g_i\in Q'$, the other case being symmetric, and consider a triangle $\Delta$ with vertices $1,g_i,g_ig_{i+1}$ and let $p\in[1, g_i]$ be a center of $\Delta$, i.e., the $\delta$-neighborhood of $p$ intersects all sides of $\Delta$.

Suppose that $p\in X-U$. Then $\dist (p,Q), \dist (p,g_i R)\leq M_1$ by relative quasiconvexity of $Q$ and $R$. By Lemma~\ref{lem:bounded-intersection}, there exists $a\in Q\cap R$ so that $\dist (p,g_ia)\leq M_2$. Since $Q'\cap R'$ is a finite index subgroup of $Q\cap R$, there is $b \in Q'\cap R'$ such that 
$\dist (p,g_ib)\leq M_3$. By Lemma~\ref{lem:computation},
$\dist (1,g_i)+\dist (1,g_{i+1})   \leq  \dist (1, g_ig_{i+1})+2M_3+2\delta.$
\begin{figure}[h]
 \includegraphics[scale=0.8]{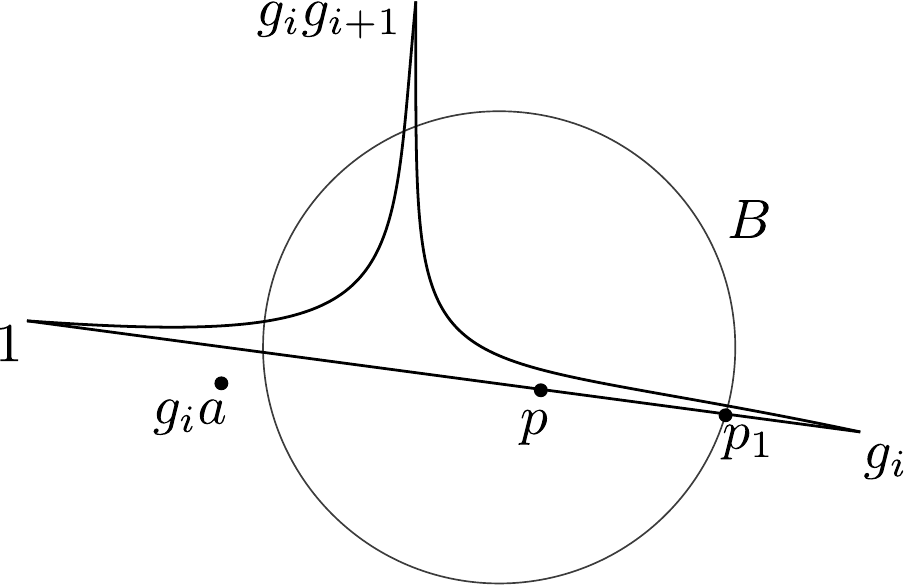}
\end{figure}

Suppose instead that $p$ is in a horoball $B$, whose stabilizer is $P$. We can assume $\dist (g_i,B)\leq M_8$. Indeed, let $p_1$ be the entrance point of the geodesic $[g_i, 1]$ in $B$; then $\dist (p_1, Q)< M_4$ by quasiconvexity of $Q$. Notice that $\dist (p_1, [g_i, g_ig_{i+1}])$ is at most $\delta$ since $p$ is a center of $\Delta$ and $p_1 \in [g_i, p]$. Notice that $\dist(p_1,[g_i,g_ig_{i+1}])$ is at most $2\delta$ (consider a triangle with vertices $p,g_i,p'$ for $p'\in [g_i,g_ig_{i+1}]$ so that $d(p,p')\leq \delta$). By quasiconvexity of $R$, there is $p_2 \in [g_i, g_ig_{i+1}]$ such that $\dist(p_1, p_2), \dist (p_2, g_iR)< M_5$.  Lemma~\ref{lem:bounded-intersection} implies  there is $a\in Q\cap R$ such that $\dist (g_ia, p_1), \dist (g_ia, p_2)<M_6$. Since $Q'\cap R'$ is a finite index subgroup of $Q\cap R$, there is $b\in Q'\cap R'$ such that $\dist (g_ib, p_1), \dist (g_ib, p_2)<M_7$.  Since $g$ can be written as $g_1\dots (g_ib)(b^{-1}g_{i+1})\dots g_n$; by minimality
\begin{equation}
\begin{split}
\dist (1, p_1) +  & \dist (p_1, g_i)  +  \dist ( g_i , p_2) + \dist (p_2, g_i g_{i+1}) = \\
& = \dist (1, g_i)+\dist(1, g_ig_{i+1}) \\
& \leq \dist (1, g_ib) +\dist (1, b^{-1}g_{i+1}) \\
&= \dist(1, p_1)+\dist(p_1, g_ib)+ \dist(g_ib, p_2) + \dist (p_2, g_ig_{i+1}),\nonumber
\end{split}
\end{equation}
and therefore 
\begin{equation}
\begin{split}
2\dist (g_i,B)  & = 2\dist (p_1, g_i) \\
 & \leq \dist (p_1, g_i) + \dist (g_i, p_2) +\dist (p_1, p_2) \\
& \leq \dist(p_1, g_ib)+ \dist(g_ib, p_2) +\dist (p_1, p_2)   \\
& \leq 2M_8.  \nonumber
\end{split}
\end{equation}

Since $Q'$ and $R'$ have compatible parabolic subgroups, assume that $Q'\cap g_i^{-1}Pg_i \leq R'\cap g_i^{-1}Pg_i $; the other case being symmetric. By quasiconvexity of $Q$, there is $q_1\in Q$ at distance $M_9$ from the entrance point of $[1, g_i]$ to $B$. 
By the parabolic approximation lemma applied to $\{1, g_i^{-1}q_1\} \subset Q'\cap \mc N_{M_9}(g_i^{-1}B)$, there is an element $a \in Q'\cap g_i^{-1}Pg_i$ such that $\dist (g_ia,q_1)\leq M_{10}$.  Observe that $a\in Q'\cap R'$. By Lemma~\ref{lem:computation}, \[\dist (1,g_i)+\dist (1,g_{i+1})   \leq  \dist (1, g_ig_{i+1})+M_{11}.  \qedhere\] 
\end{proof}

\begin{lemma}\label{lem:claim-2}
For each $i$, let $h_i=g_1\dots g_i$.  Then the concatenation $\alpha = \alpha_1\cdots \alpha_{n-1}$ of geodesics $\alpha_i$ from $h_i$ to $h_{i+1}$ is an $M'-$local $(1,K)-$quasi-geodesic for $M'= \min\{ \dist (1, g_i) \}$. 
\end{lemma}
\begin{proof}
This holds in view of Lemma~\ref{lem:claim} and the following computation for $x\in [h_{i-1},h_{i}]$ and $y\in [h_{i},h_{i+1}]$:
\begin{equation} 
\begin{split}
\dist (h_{i-1},x) & +\dist (x,y)  +\dist (y,h_{i+1})\geq  \dist (h_{i-1},h_{i+1}) \geq \\
& \geq \dist (h_{i-1},h_i)+\dist (h_i,h_{i+1})-K=\\
& = \dist (h_{i-1},x)+\dist (x,h_i)+\dist (h_i,y)+\dist (y,h_{i+1})-K  \nonumber \end{split}
\end{equation}
that yields $\dist (x,y)+K \geq \dist (x,h_i)+\dist (h_i,y)$.
\end{proof}

\bibliographystyle{plain}
\bibliography{Xbib.bib}

\end{document}